\documentclass{article}
\usepackage{graphicx}
\usepackage{epsfig}
\usepackage{amsmath}
\usepackage{amssymb}
\usepackage{amsthm}
\usepackage{stmaryrd}
\usepackage{url}
\usepackage{longtable}
\usepackage[figuresright]{rotating}
\usepackage{color}
\usepackage{amsfonts}
\usepackage{amsthm}
\usepackage{mathrsfs}
\usepackage{stmaryrd}
\usepackage{wasysym}
\usepackage{xypic}
\xyoption{curve}
\usepackage{cases}
\usepackage{hyperref}
\usepackage{enumerate}
\usepackage{hyperref}
\usepackage{ulem}
\usepackage{paralist}
\newcommand{\Proj}[2]{\mathbb{P}^{#1}_{#2}} 
\newcommand{\Field}[1]{#1}			
\newcommand{\Ring}[1]{#1}			
\newcommand{\SL}[2]{\mathrm{SL}_{#1}(#2)}    
\newcommand{\GL}[2]{\mathrm{GL}_{#1}(#2)}		
\newcommand{\Q}{\mathbb{Q}}            
\newcommand{\F}[1]{\mathbb{F}_{#1}}    
\newcommand{\pIdeal}[1]{\mathfrak{p}_{\tiny\!#1}}
\newcommand{\pElem}{\varpi}
\newcommand{\Units}{\mathrm{U}}
\newcommand{\oRing}[1]{\mathfrak{o}_{\tiny\!#1}}
\newcommand{\val}{\mathrm{v}} 				
\newcommand{\galConj}[1]{\overline{#1}}			
\newcommand{\relNorm}[2]{\mathrm{N}_{{\hspace{0.5pt}\!_{#1}}\hspace{0.5pt}}\!\left(#2\right)}
\newcommand{\relTrace}[2]{\mathrm{Tr}_{\hspace{0.5pt}\!_{#1}}\!\left(#2\right)}
\newcommand{\NormOne}{\NormFib{1}}
\newcommand{\NormFib}[1]{\mathrm{N}^{{#1}}}
\newcommand{\ramInd}[1]{\mathrm{e}_{#1}}	 	
\newcommand{\diff}[1]{\delta\!\left(#1\right)} 				
\newcommand{\Mat}[1]{\mathtt{#1}}		
\newcommand{\matW}{\mathtt{w}}		
\newcommand{\matP}{\mathtt{p}}			
\newcommand{\qElem}{\eElem} 				
\newcommand{\trZeroMat}{\tau}
\newcommand{\qPnt}{\alpha}				

\newcommand{\Tree}{\mathcal{T}}     
\newcommand{\conj}[2]{#1^{#2}}  		
\newcommand{\smVertex}[3]{\left[\begin{smallmatrix}#1\\#2\end{smallmatrix}\right]_{#3}}
\newcommand{\Id}[1]{1_{{#1}}} 					

\newcommand{\eElem}{\lambda}  					
\newcommand{\mElem}{\alpha}					
\newcommand{\ceiling}[1]{\left\lceil#1\right\rceil} 
\newcommand{\floor}[1]{{\left\lfloor#1\right\rfloor}} 

\renewcommand{\mod}[1]{{\hspace{2pt}(\mathtt{mod}\hspace{4pt}#1)}}
\newcommand{\grp}[1]{\mathrm{#1}}						
\newcommand{\G}{\mathrm{G}}							
\newcommand{\Torus}{\grp{T}}							
\newcommand{\I}[1]{\grp{I}_{#1}}						
\newcommand{\K}[1]{\grp{K}_{#1}}						
\newcommand{\J}[1]{\grp{J}_{#1}}						
\theoremstyle{plain}

\newtheorem*{theorem*}{Theorem}
\newtheorem{lemma}{Lemma}
\newtheorem*{lemma*}{Lemma}

\theoremstyle{definition}

\newtheorem*{definition*}{Definition}
\newtheorem*{notation*}{Notation}

\theoremstyle{comment}
\newtheoremstyle{named}{}{}{\itshape}{}{\bfseries}{.}{.5em}{#1 #3}
\theoremstyle{named}

\title{Stabilizers of quadratic points on the Bruhat-Tits tree of $\mathrm{SL}(2)$ over finite extensions of $\Q_2$}
\author{Terence Joseph K\i vran-Swaine\\
\texttt{tjswaine@gmail.com}}
\begin{document}
\maketitle
\begin{abstract}For $\Field{F}$, a finite extension of $\Q_2$, and $\Field{E}$ a quadratic extension of $\Field{F}$, I compute the stabilizer in $\SL{2}{\Field{F}}$ of a point in the Bruhat-Tits tree of $\SL{2}{\Field{E}}$.\end{abstract}
\paragraph{Introduction.}
\paragraph{Introduction.}
This paper is a first step towards constructing an explicit, uniform parametrization of the irreducible, admissible representations of $\grp{G}=\SL{2}{\Field{F}}$ where $\Field{F}$ is a finite extension of $\Q_2$ by means similar to that for $\GL{2}{\Field{F}}$ found in \cite{BH2006}.\\
\indent In his Ph.D. thesis, \cite{Sh1966}, Joseph Shalika establishes that irreducible subrepresentations of the representation of Weil (\cite{W1964}) include an exhaustive list of irreducible, cuspidal representations over $\SL{2}{\Q_p}$ for $p\neq 2$, by classifying such representations by their restriction to particular maximal compact subgroups of $\mathrm{SL}_2$.   In his 1972 paper, \cite{C1972}, Casselman extends Shalika's technique to account for the irreducible cuspidal representations which occur in the construction of Weil for $\SL{2}{\Field{F}}$ of even residual characteristic. In 1976 \cite{NW1976}, Nobs and Wolfart construct an exhaustive list of ``exceptional" representations of $\grp{G}$ that exist outside the construction of Weil as subrepresentations of tensor products of representations in the Weil representation.\footnote{ In \cite{N1978}, Nobs presents the four conjugacy classes of the induction of these representations to $\GL{2}{\Q_2}$.  Such exceptional represntations on $\GL{2}{\Field{F}}$ were shown in 1990 to restrict to irreducible, exceptional representations of $\grp{G}$ by Kutzko and Pantoja \cite{KP1990}.}\\
\indent In his 1978 papers, \cite{K1978p1} and \cite{K1978p2}, Kutzko builds on his thesis \cite{K1972}\footnote{This was announced in \cite{K1973}.} to explicitly construct cuspidal representations of $\GL{2}{\Field{F}}$ for arbitrary residual characteristic.  Such an examination of representations by restriction to compact subgroups had been suggested by Howe in 1977 (\cite{Ho1977}).   This parametrization became know as the theory of cuspidal types or the theory of minimal $\mathsf{K}$-types.  This technique was adapted to $\SL{2}{\Field{F}}$ for $F$ of odd residual characteristic by Manderscheid in his 1984 papers, \cite{M1} and \cite{M2}.  In 1994 Moy and Prasad proved that a more general approach for classifying irreducible representations could be employed for arbitrary reductive groups over arbitrary $p$.  Resently such approaches have been employed in work by Yu, \cite{Y2001} and in the exposition of the harmonic analyis of $\SL{2}{\Field{F}}$ for $p\neq 2$ by Adler, Debacker, Sally and Spice, \cite{ADSS2011}.\\
\indent The aforementioned constructions all depend on constructing certain representations of compact open subgroups of $\grp{G}$ which are stabilized in $\grp{G}$ only by elements in their domain.  Such representations are constructed by extending dimension one representations on abelian subgroups or quotients of subgroups to their $\grp{G}$-stabilizers.  In order to explicitly compute stabilizers of these characters of compact open subgroups of $\grp{G}$, I compute the stabilizers of quadratic points over the Bruhat-Tits tree of $\grp{G}$ the action on which is closely related to the action on the characters of compact open subgroups via the eigenspaces of trace-zero matrices.  This is stated as the Theorem on page \pageref{thm:stab}.  To facilitate smoother reading of this note I have included proofs of a couple of more elementary facts that I use in the Appendix.

\paragraph{Notation.}
Here, $n$ denotes a non-negative integer and $\Proj{1}{\Ring{A}}$ is the projective line over a commutative ring $\Ring{A}$.\\
\indent For a real number $x$, $\floor{x}$ will denote the greatest integer $\leq x$ and $\ceiling{x}$ will denote the least integer $\geq x$.\\
\indent Let $\Field{F}$ be an algebraic extension of $\Q_2$ of ramification $\ramInd{}=\ramInd{\Field{F}}$ with integer ring $\oRing{\Field{F}}$, prime ideal $\pIdeal{\Field{F}}$, local uniformizing parameter (l.u.p.) $\pElem:=\pElem_{\Field{F}}$, valuation map $\val_{\Field{F}}$ and unit filtration $\Units_{\Field{F}}^0:=\Units_{\Field{F}}=\oRing{\Field{E}}^{\times}$, $\Units^n_{\Field{F}}:=\{x\in\Units_{\Field{F}}|x\equiv1 \mod{\pIdeal{\Field{F}}^n}\}$ for $n>0$.
\\
\indent For a group of matrices $\grp{H}$ and an invertible matrix $\Mat{A}$, let $\conj{\grp{H}}{\Mat{A}}$ be the conjugate $\Mat{A}\grp{H}\Mat{A}^{-1}$.  Of particular importance will be conjugation by $\matP:=\left(\begin{smallmatrix}
0&1\\\pElem_{\Field{F}}&0
\end{smallmatrix}\right)$. 
\\
\indent Let $\grp{G}(\Field{F})=\SL{2}{\Field{F}}$, with maximal compact subgroup $\grp{K}=\grp{K}(\Field{F}):=\SL{2}{\oRing{\Field{F}}}$, 
and Bruhat-Tits tree $\Tree_{\Field{F}}$\cite[Chapter 2]{S1977}.\\
\indent Let $\Field{E}/\Field{F}$ be a quadratic extension of ramification index $\ramInd{\Field{F}/\Field{E}}$.  I will use notation for $\Field{E}$ similar to $\Field{F}$.  For $z\in\Field{E}$, let $z\mapsto\galConj{z}$ denote the conjugate map, $\mathrm{Tr}$ denote the trace map, $\mathrm{N}$ denote the norm map, and $\delta$ denote the different map ({\it c.f. }\cite[Chapter 3]{N1999}), all with respect to $\Field{E}/\Field{F}$.  
 Finally define the following matrix, with coefficients in $\oRing{\Field{F}}$,
\begin{equation}
\matP_{\Field{E}}:=
\begin{cases}
\quad\pElem_{\Field{F}}1_2 &\mbox{if } \ramInd{\Field{F}/\Field{E}}=1\\\
\left(\begin{smallmatrix}\relTrace{}{\pElem_{\Field{E}}}&1\\- \relNorm{}{\pElem_{\Field{E}}}&0\end{smallmatrix}\right) & \mbox{otherwise.}
\end{cases}
\end{equation}

\paragraph{Quadratic points over the Bruhat-Tits tree.}  If $\Field{E}/\Field{F}$ is ramified, it is totally, wildly ramified.  As is observed in \cite{T1979}, an injection of fields $\Field{F}\rightarrow\Field{E}$ of finite index naturally induces a topological injection of trees $\Tree_{\Field{F}}\hookrightarrow\Tree_{\Field{E}}$ which under the graph metric enjoys a global dilation factor of $\ramInd{\Field{E}/\Field{F}}$.\footnote{The structure of the injection of trees is established in \autoref{lem:Trees} in the appendix.  For a fascinating set of examples of this phenomenon see \cite{CK2005}.}  Moreover the Galois action on $\Field{E}/\Field{F}$ induces one on $\Tree_{\Field{E}}$.  Details of this relationship can be found in the appendix.   A \emph{quadratic point over  $\Tree_{\Field{F}}$} is the orbit of a vertex in the difference of sets $\Tree_{\Field{E}}\setminus\Tree_{\Field{F}}$ under conjugation for some quadratic extension, $\Field{E}/\Field{F}$. The \emph{level} of a quadratic point is the distance in $\Tree_{\Field{E}}$ from that orbit to the vertices of $\Tree_{\Field{F}}$.
\paragraph{Remark.}Quadratic points can also be characterized as points in the Berkovich space of $\Proj{1}{\Field{F}}$ which valuate $\Field{F}$-rational functions over a sufficiently small ball centered at a point in $\Proj{1}{\Field{E}}\setminus\Proj{1}{\Field{F}}$.  That the latter characterization coincides with the former is clear since the multiplicity of a zero or pole for a $\Field{F}$-rational function at a point $z\in\Proj{1}{\Field{E}}$ is equal to the multiplicity at its conjugate.  \footnote{This second characterization owes thanks to Xander Faber for sharing with me his paper \cite{F2013}, which is an excellent introductory reference on Berkovich space.  For connections between Berkovich space and the Bruhat-Tits tree see the 2007 Arizona Winter School notes \cite{DT2007}.}
\paragraph{Eigenspaces and elliptic tori.} A vertex on a Bruhat-Tits tree of $\SL{2}{\Field{E}}$ can be identified with a ball in $\Proj{1}{\Field{E}}$ and consequently a set of one-dimensional subspaces of $\Field{E}\oplus\Field{E}$.  Via this identification, a quadratic point over $\Tree_{\Field{F}}$ is a class of pairs of distinct, conjugate dimension-one subspaces.  For each pair of conjugate subspaces, $\qPnt$, there is an elliptic torus, $\grp{T}_\qPnt$,  in  $\SL{2}{\Field{F}}$ the elements of which share those two subspaces as eigenspaces.  Denote by $[\qPnt]_n$ the quadratic point over $\Tree_{\Field{F}}$, identified with set of points in $\Proj{1}{\Field{E}}=\Proj{1}{\oRing{\Field{E}}}$ which share an image with $\qPnt$ under the natural map $\Proj{1}{\oRing{\Field{E}}}\rightarrow\Proj{1}{\Ring{R}}$ where $\Ring{R}=\oRing{\Field{E}}/\pIdeal{\Field{E}}^{n}$.
%
%
\begin{theorem*}\label{thm:stab} Let $\Field{F}$ be a finite extension of $\Q_2$ with ramification index $\ramInd{}$.  Let $[\qPnt]_{\bullet}$ be a quadratic point over $\Tree({\Field{F}})$ of level $n\ramInd{\Field{E}/\Field{F}}$ in $\Tree({\Field{E}})$.\\
Matrices which stabilize $[\qPnt]_{\bullet}$ in $\G$ are conjugate to the form $\Mat{t}\cdot\Mat{j}\cdot\Mat{s}$, where $\Mat{t}\in\grp{T}_{\qPnt}$, $\Mat{j}=\left(\begin{smallmatrix}
1+\pElem^{n-m} a& \pElem^{n-\varepsilon}b\\
 \pElem^{n}c & 1+\pElem^{n-m} d\\
\end{smallmatrix}\right)\in\grp{G}$ and $\Mat{s}=1_2$, if $\Field{E}/\Field{F}$ is unramified, and $\Mat{s}=\left(\begin{smallmatrix}1&0\\0&z\end{smallmatrix}\right)(x1_2+y\matP_{\Field{E}})$, otherwise.\\
Here
$x\in \{1\}\cup\Units_{\Field{F}}^{n-\ceiling{\frac{t}{2}}+\partial}\setminus\Units_{\Field{F}}^{n-m}$,  
$y \in  \{0\}\cup\pIdeal{\Field{F}}^{n-\floor{\frac{t}{2}}+1-\partial}\setminus\pIdeal{\Field{F}}^{n-1}$
and $z=\relNorm{}{x+y\pElem_{\Field{E}}}^{-1}$,\\
and $m=\min(\floor{\frac{n}{2}},\ramInd{})$, $t=\min(n,\val_{\Field{E}}(\diff{\pElem_{\Field{E}}})-1)$, $\varepsilon=\ramInd{\Field{E}/\Field{F}}-1$ and $\partial= \ceiling{(\ramInd{\Field{E}}+1-\val_{\Field{E}}(\diff{\pElem}))/\ramInd{\Field{E}}}$.
\end{theorem*}%
\paragraph{Filtrations of $\grp{K}$ and $\grp{I}$.} In addition to the notation above I will also make use of a filtrations on $\grp{K}$ constructed as follows.  For subgroups $\grp{A}$ and $\grp{B}$ of $\grp{G}$, $\grp{A}\ast\grp{B}$ is the smallest subgroup of $\grp{G}$ containing $\grp{A}$ and $\grp{B}$.  Denote by $\phi_n$ the natural map from $\oRing{}$ to $\oRing{}/\pIdeal{}^n$ and denote the induced map,   

\begin{equation}
\varphi_n:\SL{2}{\oRing{\Field{F}}}\rightarrow\SL{2}{\oRing{\Field{F}}/\pIdeal{\Field{F}}^n}.
\end{equation}

 With $n$ a non-negative integer, set $m=\min(\floor{\frac{n}{2}},\ramInd{\Field{F}})$.  Let $\varepsilon\in\{0,1\}$.  Define $\grp{K}_{n}:=\ker(\varphi_n)$, $\grp{I}_{2n}:=\K{n}\cap\conj{\K{n}}{\matP}$,  $\grp{I}_{2n+1}:=\K{n+1}\ast\conj{\K{n+1}}{\matP}$, $\grp{J}_{0}:=\grp{K}$,
\begin{equation}
\begin{split}
&\grp{J}_{n}:=\left\{\left(\begin{smallmatrix}u_1&b\\c&u_2
\end{smallmatrix}\right):u_1,u_2\in\Units^{n-m}_{\Field{F}}; b,c\in\pIdeal{\Field{F}}^n; u_1u_2-bc=1 \right\}\text{, if $n>0$, and }\\
&\J{2n+\varepsilon}^{\mathtt{r}}:=\left\{\left(\begin{smallmatrix}u_1&b\\c\pElem_{\Field{F}}&u_2
\end{smallmatrix}\right):u_1,u_2\in\Units^{n-m+\varepsilon}_{\Field{F}}; b,c\in\pIdeal{\Field{F}}^{n}; u_1u_2-bc\pElem_{\Field{F}}=1 \right\}.
\end{split}
\end{equation}
One can quickly observe that  $\J{2n}^{\mathtt{r}}=\J{n}\cap\conj{\J{n}}{\matP}$, that $\J{2n+1}^{\mathtt{r}}=\J{n+1}\ast\conj{\J{n+1}}{\matP}$, and, if $n>1$, $\J{n}\subsetneq\K{n}$ and $\J{n+1}^{\mathtt{r}}\subsetneq\I{n+1}$.

\paragraph{The action of $\grp{G}$ on $\Tree_{\Field{E}}$.}
The vertices of $\Tree_{\Field{F}}$ and hence their image in $\Tree_{\Field{E}}$ are in bijection with the maximal compact subgroups of $\grp{K}$ which stabilize them.  With respect to $\grp{G}$ these form two orbits, represented by $\grp{K}$ and $\conj{\grp{K}}{\matP}$.  It will be useful to coordinatize $\Tree_\Field{E}$ with respect to the $\grp{K}$-stabilized vertex.\\
\indent Let $\smVertex{0}{0}{0}$ be the vertex of $\Tree_{\Field{E}}$ which is stabilized by $\grp{K}(\Field{F})$ (and by $\grp{K}({\Field{E}})$).\\
\indent Let $\smVertex{x}{y}{n}$ be the vertex of $\Tree_{\Field{E}}$ which is $n$ edges away from $\smVertex{0}{0}{0}$ and corresponds to $\smVertex{x}{y}{}\in\Proj{1}{\Ring{R}}$ for $\Ring{R}=\oRing{\Field{E}}/\pIdeal{\Field{E}}^{n}$. The natural action of $\grp{K}(\Field{E})$ on each such $\Proj{1}{\Ring{R}}$ corresponds to that on such vertices in $\Tree_{\Field{E}}$.  Extending this correspondence $\conj{\grp{K}}{\matP}$ will stabilize $\smVertex{0}{1}{\ramInd{\Field{E}/\Field{F}}}$, however its action on $\Tree_{\Field{E}}$ will not stabilize the set of vertices corresponding to the set $\Proj{1}{\Ring{R}}$, so its action, in this respect, is incompatible with the coordinates chosen nor is the action of any other maximal compact subgroup.  This notation will be used to compute the stabilizer of a quadratic point that is closest to $\smVertex{0}{0}{0}$ if $\Field{E}/\Field{F}$ is unramified or to $\smVertex{0}{1}{1}$ otherwise.  The stabilizers of other points can then be computed via conjugation by elements of the set $\grp{G}\cup\matP\grp{G}$.\\
\indent Let $\grp{B}_n$  denote the preimage of the upper triangular matrices under $\varphi_n$.  This is the $\grp{K}$-stabilizer of the vertex $\smVertex{0}{1}{n\ramInd{\Field{E}/\Field{F}}}$.
\paragraph{Barbs.}A \emph{barb}\footnote{This terminology is due to Tits. \cite{T1979}}  in $\Tree_{\Field{E}}$ with respect to $\Field{F}$ is a connected set of points in $\Tree_{\Field{E}}\setminus\Tree_{\Field{F}}$ which are fixed by Galois conjugation.  Using the above coordinatization of $\Tree_{\Field{E}}$ it is clear that the diameter of barbs is $\ramInd{\Field{E}}-1$ unless there is a trace-zero l.u.p. in $\Field{E}$, in which case it is $\ramInd{\Field{E}}$.  In the case that $\Field{E}/\Field{F}$ is wildly ramified, the Galois-fixed subtree is strictly larger than the image of $\Tree_{\Field{F}}$ and the in the terminology of Tits \cite[page 47 subsection 2.6]{T1979}, the tree is ``covered in barbs'', however, by considering orbits of verteces, the following calculations will be unaffected by this phenomenon.

\begin{lemma}\label{lem:J} For $n\geq 0$, under the $\grp{K}_{n-m}$-action on $\Tree_{\Field{E}}$, 
the point-wise stabilizer of the ball centered at $\smVertex{0}{0}{0}$ of radius $n\ramInd{\Field{E}/\Field{F}}$ is
$\grp{J}_{n}$ and the point-wise stabilizer of the ball centered at $\smVertex{0}{1}{1}$ of radius $(n\ramInd{\Field{E}/\Field{F}}+1)$ is $\J{2n+1}^{\mathtt{r}}$ if $\Field{E}/\Field{F}$ is ramified and $\conj{\J{n+1}}{\matP}$ otherwise. \end{lemma}
\begin{proof}
From the observations that $\J{2n+1}^{\mathtt{r}}=\J{n+1}\ast\conj{\J{n+1}}{\matP}$, this can be reduced to the unramified case.  In the ramified case the ball of radius $2n+1$ centered at $\smVertex{0}{1}{1}$ is the intersection of the balls of radius $2n+2$ centered at $\smVertex{0}{0}{0}$ and $\smVertex{0}{1}{2}$. Since that if the image of $\Tree_{\Field{F}}$ enjoys no new vertices in $\Tree_{\Field{E}}$, if $\Field{E}/\Field{F}$ is unramified, $\matP$ maps $\smVertex{0}{0}{0}$ to $\smVertex{0}{1}{1}$, so it suffices to establish that the stabilizer of the ball centered at $\smVertex{0}{0}{0}$ of radius $n\ramInd{\Field{E}/\Field{F}}$ is
$\grp{J}_{n}$.\\
Since $\grp{B}_{n}$ is the stabilizer in $\grp{K}$ of the vertex $\smVertex{0}{1}{n\ramInd{\Field{E}/\Field{F}}}$, it is enough to show that $\grp{J}_{n}$ is the intersection of the $\grp{K}(\Field{E})$ conjugates of $\grp{B}_{n}\cap\grp{K}_{n-m}$.  With $\matW=\left(\begin{smallmatrix}0&1\\-1&0\end{smallmatrix}\right)$, note that  $\grp{J}_n=\conj{\grp{B}_{n}}{\matW}\cap\grp{K}_{n-m}\cap\grp{B}_{n}$,  
where $\conj{\grp{B}_{n}}{\matW}$ is precisely the group of determinant-one matrices of the form
$
\left(\begin{smallmatrix}
\oRing{}^\times&\pIdeal{}^{n}\\
\oRing{}&\oRing{}^\times
\end{smallmatrix}\right).
$
 To show that it is a subset of the other conjugates of $\grp{B}_{n}\cap\grp{K}_{n-m}$ and hence the intersection of all conjugates of $\grp{B}_{n}\cap\grp{K}_{n-m}$,   I compute the the action of $\grp{J}_{n}$ on an arbitrary point of the projective line over $\oRing{\Field{E}}/\pIdeal{\Field{E}}^{n\ramInd{\Field{E}/\Field{F}}}$, other than $\smVertex{0}{1}{n\ramInd{\Field{E}/\Field{F}}}$, illustrating that it is trivial. Since $\ramInd{}, \floor{\frac{n}2} \geq m$,
$\left(\begin{smallmatrix}
1+\pElem^{n-m} a& \pElem^{n}b\\
 \pElem^{n}c & 1+\pElem^{n-m} d\\
\end{smallmatrix}\right)
\smVertex{x}{1}{n\ramInd{\Field{E}/\Field{F}}}\!\!\!\!=
\smVertex{(1+\pElem^{n-m} a)^2x}{1}{n\ramInd{\Field{E}/\Field{F}}}
\!\!\!\!= \smVertex{x}{1}{n\ramInd{\Field{E}/\Field{F}}}.$

\end{proof}
\paragraph{Kernels.} Let $\NormOne$ represent the kernel of $\mathrm{N}$.\\
\indent I will make us of the following subgroup of $\GL{2}{\oRing{\Field{F}}}$,
$\grp{K}_{(\!n\!)}\!\!:=\!\ker(\phi_n\!\circ\mathrm{det})$.
 Since $\grp{K}_n\!\unlhd\grp{K}_{(\!n\!)}$, there is an inclusion 
$\iota_n\!\!:\!\grp{K}/\grp{K}_n\!\! \hookrightarrow\grp{K}_{(\!n\!)}/\grp{K}_n$.\\
\indent For $\trZeroMat$, a trace zero matrix with coefficients in $\oRing{\Field{F}}$, define $\grp{T}_{\trZeroMat}$ and $\tilde{\grp{T}}_{\trZeroMat}$  to be the elliptic torus generated by $\trZeroMat$ in $\grp{G}$ and $\GL{2}{\oRing{\Field{F}}}$, respectively.\\ 
\indent Set $\grp{T}_{\trZeroMat, n}:=\ker(\phi_n\circ\mathrm{N}\circ\iota_n\circ\varphi_n)$, where $\mathrm{N}$ is the restriction of determinant in  $\GL{2}{\oRing{\Field{F}}}$ to $\tilde{\grp{T}}_{\trZeroMat}$. \\ 
\indent A noted above, given a conjugate pair of one dimensional vector spaces $\alpha$ in $\Proj{1}{\Field{E}}$ one can consider an of isomorphism of $\Field{E}\oplus\Field{E}$ which enjoy $\alpha$ as eigenspaces and which are defined over $\Field{F}$.   From such an isomorphism another of trace zero, $\trZeroMat$ can be constructed as a linear combination of the original and the identity map.  Since the action on this $\alpha$ by elemets of $\GL{2}{\Field{F}}$ is adjoint to the action on such isomorphisms by conjugation,  the stabilizer of an $\alpha$ can be identified with that of a corresponding $\trZeroMat$.  Through this identification, any $\grp{T}_\alpha$ is isomorphic to some $\grp{T}_\tau$.\\
\indent  I extend the definition $\grp{T}_{\alpha, n}:=\grp{T}_{\trZeroMat, n}$ where the eigenvectors of $\trZeroMat$ are $\alpha$.  Observe that $[\qPnt]_{n\ramInd{\Field{E}/\Field{F}}}$ is fixed by $\grp{T}_{\alpha, n}$.

\paragraph{Quadratic extensions of $\Field{F}$.}  The quadratic extensions, $\Field{E}$, of $\Field{F}$ can be classified by the different (ideal) of each extension.  For the unramified extension, the different is the ring of integers.  For ramified extensions, the different is generated by $\diff{\pElem_{\Field{E}}}=\pElem_{\Field{E}}-\galConj{\pElem_{\Field{E}}}$ \cite[Chapter 3]{N1999}.  The most extreme, non-trivial different is therefore $\pIdeal{\Field{E}}^{2\ramInd{}+1}$.  In that case, $\Field{E}$ is the splitting field for  $x^2-\pElem_{\Field{F}}$ for some choice of $\pElem_{\Field{F}}$.   Otherwise, $\Field{E}$ is the splitting field of $x^2+ax+b\pElem_{\Field{F}}$, $b\in\Units_{\Field{F}}$, and $\val_{\Field{F}}(a)=\val_{\Field{E}}(\diff{\pElem_{\Field{E}}})/2=d\leq\ramInd{}$.  These claims are stated as \autoref{lem:Eisnstn} and proven in the appendix.\footnote{For a given value of $\val_{\Field{E}}({\diff{\pElem_{\Field{E}}}})$ the extensions can be enumerated using Artin-Schreier theory and have been so in an unpublished paper by Wan and Moreno , \cite{WM2005}.}  Those extensions without extremal different therefor split the polynomial $x^2-(a^2-4b\pElem_{\Field{F}})\pElem_{\Field{F}}^{-2d}$, the roots of which are contained in $\Units_{\Field{E}}^{2(\ramInd{}-d)+1}$.%
\begin{lemma}\label{lem:LF}\begin{inparaenum}[\upshape(\itshape a\upshape)]\item
If $\Field{E}/\Field{F}$ is a ramified, quadratic extension  then $\mathrm{N}^{-1}(\Units_{\Field{F}}^n)=\NormOne\Units_{\Field{E}}^{2n-t}$ where $t=\min(n,\val_{\Field{E}}(\diff{\pElem_{\Field{E}}})-1)$ and\\
\item  if $\Field{E}/\Field{F}$ is an unramified, quadratic extension then $\mathrm{N}^{-1}(\Units_{\Field{F}}^n)=\NormOne\Units_{\Field{E}}^{n}$.
\end{inparaenum}
\end{lemma}
\begin{proof} The ramified case follows from Chapter V, \S 3 , Proposition 5 of  \cite{S1979} and the unramified from Proposition 3 ({\it loc. cit.}).\end{proof}
The following lemma also follows from Propositions 5 and 3.
\begin{lemma}[Casselman, \cite{C1972}]\label{lem:Cslmn}\footnote{I have omitted part \upshape(\itshape c\upshape) of this lemma as I do not employ it in this paper.}
\begin{inparaenum}[\upshape(\itshape a\upshape)]\item If $\Field{E}/\Field{F}$ is ramified, $\NormOne\subseteq \Units_{\Field{E}}^{\val_{\Field{E}}(\delta( \pElem_{\Field{E}}))-1}$, and $\NormOne/\NormOne\cap\Units_{\Field{E}}^{\val_{\Field{E}}(\delta( \pElem_{\Field{E}}))}$ is a group of order 2.\\
\item Also if $\Field{E}/\Field{F}$ is ramified, 
\begin{equation*} 
\NormOne\cap\Units_{\Field{E}}^{\val_{\Field{E}}(\delta( \pElem_{\Field{E}}))+2n-1}=
\NormOne\cap\Units_{\Field{E}}^{\val_{\Field{E}}(\delta( \pElem_{\Field{E}}))+2n},\quad n\geq1. \end{equation*}
\end{inparaenum}
\end{lemma}
\paragraph{Proof of Theorem.}  Since the statement of the theorem is up to conjugation, I can assume without loss of genreality, that the quadratic point in question is closest of the points in $\Tree_{\Field{F}}$ to either $\smVertex{0}{0}{0}$  or $\smVertex{0}{1}{\ramInd{\Field{E}/\Field{F}}}$, and that the distance between it an those points is $n$.  If $\ramInd{\Field{E}/\Field{F}}=1$, and the point is closer to $\smVertex{0}{1}{1}$ conjugation by $\matP$ reduces the problem to the the case closest to $\smVertex{0}{0}{0}$.  As a result I can further assume $[\qPnt]_{\bullet}= [\qPnt]_{n}$.\\
\indent Let $\grp{J}^\prime$ equal $\grp{J}_n$ if $\ramInd{\Field{E}/\Field{F}}=1$ and $\grp{J}_{2n-1}^{\mathtt{r}}$ otherwise. Observe that if $\Field{E}$ is ramified that conjugation by $\matP_{\Field{E}}$ stabilizes $\smVertex{0}{1}{1}$ and, consequently, any stabilizer in $\GL{2}{\Field{F}}$ of quadratic points closest to $\smVertex{0}{1}{1}$ are congruent to a linear combination of $\matP_{\Field{E}}$ and $1_2$ modulo $\grp{J}^\prime$.\\
\indent Since $\grp{T}_{\qPnt,n}$ and $\grp{J}^\prime$ stabilize  $[\qPnt]_{n}$ I compute the possible coefficients of $\grp{T}_{\qPnt,n}\grp{J}^\prime$.  By \autoref{lem:LF} part \upshape(\itshape a\upshape), elements of $\grp{T}_{\qPnt,n}$ are of the form $\Mat{t}\cdot\Mat{k}$ where $\Mat{t}\in\grp{T}_\qPnt$ and $\Mat{k}=\left(\begin{smallmatrix}1&0\\0&c\end{smallmatrix}\right)( a1_2+b\Mat{p}_{\Field{E}})$ where $a\in \Units_{\Field{F}}^{n-\ceiling{\frac{t}{2}}}$,  $b \in  \pIdeal{\Field{F}}^{n-\floor{\frac{t}{2}}}$ and $c =\relNorm{}{a+b\pElem_{\Field{E}}}^{-1}$.  By \autoref{lem:Cslmn} if there is no trace-zero l.u.p., then the restriction can be strengthened to $a\in \Units_{\Field{F}}^{n-\ceiling{\frac{t}{2}}+1}$ and if there is, to  $b \in  \pIdeal{\Field{F}}^{n-\floor{\frac{t}{2}}+1}$.  If $a\in\Units_{\Field{F}}^{n-m}$ and if $b\in\pIdeal{\Field{F}}^{n-1}$, then $\left(\begin{smallmatrix}1&0\\0&c\end{smallmatrix}\right)( a1_2+b\Mat{p}_{\Field{E}})\in\grp{J}_{2n-1}^{\mathtt{r}}$.  By \autoref{lem:LF} part \upshape(\itshape b\upshape) if $\Field{E}/\Field{F}$ is unramified, $\grp{T}_{\qPnt, n}\subset\grp{T}_{\qPnt}\grp{J}_{n}$.  As a result, there is a factorization of $\grp{T}_{\qPnt,n}\grp{J}^\prime$ of the form $\Mat{t}\cdot\Mat{j}\cdot\Mat{s}$ with $\Mat{j}\in \grp{J}^{\prime}$ and $\Mat{s}$ as mentioned in the theorem.\\
\indent The theorem has been reduced to showing in the ramified case that the stabilizer of $[\qPnt]_{n\ramInd{\Field{E}/\Field{F}}}$ is contained in $\grp{T}_{\qPnt,n}\grp{J}^{\prime}$.  This is proven below.
\pagebreak[1]
\begin{lemma}\label{lem:Tan}
If no point in $\Tree_{\Field{F}}$ is closer to $[\qPnt]_{n\ramInd{\Field{E}/\Field{F}}}$ than $\smVertex{0}{0}{0}$, the stabilizer of $[\qPnt]_{n\ramInd{\Field{E}/\Field{F}}}$ is contained in $\Torus_{\qPnt, n}\grp{J}^\prime$, where $\grp{J}^\prime$ is $\grp{J}_n$ if $\ramInd{\Field{E}/\Field{F}}=1$ and $\grp{J}_{2n-1}^{\mathtt{r}}$ otherwise. 
\end{lemma}
\begin{proof}
In the argument below I construct a matrix of the form $\Mat{A}=( x1_2+y\trZeroMat+z\left(\begin{smallmatrix}0&0\\1&0\end{smallmatrix}\right))\in\grp{G}$ where $x,y\in\Field{F}$ and $\trZeroMat$ is a trace-zero matrix with eigenvalues $\pm\qElem\in\oRing{\Field{E}}\setminus\oRing{\Field{F}}$ and eigenvectors which represent $[\qPnt]_{n\ramInd{\Field{E}/\Field{F}}}$.  Assuming $\Mat{A}$ stabilizes I then calculate the restrictions on the coefficients $x$, $y$ and $z$ in terms of $x^2+y^2|\trZeroMat|$.\\  
\indent First I will handle when $\val_{\Field{E}}(\diff{\pElem_{\Field{E}}})=\ramInd{\Field{E}}+1$. Choose a local uniformizing parameter, $\eElem:=\pElem_{\Field{E}}$, such that, $\relTrace{\Field{E}/\Field{F}}{\pElem_{\Field{E}}}=0$.  Set $\pElem:=\pElem_{\Field{F}}=\pElem_{\Field{E}}^2$ and $\trZeroMat:=\matP_{\Field{E}}$.\\
\indent Assume $\Mat{A}:=\left(\begin{smallmatrix}u&x\\ (x+y)\pElem&\frac{1+x(x+y)\pElem)}{u}\end{smallmatrix}\right)$ stabilizes $\left[\begin{smallmatrix}1\\ \eElem\end{smallmatrix}\right]_{n\ramInd{\Field{E}/\Field{F}}}$.
\begin{equation}
\Mat{A}\left[\begin{smallmatrix}1\\\eElem\end{smallmatrix}\right]=\left(\begin{smallmatrix}u&x\\(x+y)\pElem&(1+x(x+y)\pElem)/u\end{smallmatrix}\right)\left[\begin{smallmatrix}1\\\eElem\end{smallmatrix}\right]
=\left[\begin{smallmatrix}1\\ \left((x+y)\pElem+(1+x(x+y)\pElem)\eElem/u\right)(u+x\eElem)\end{smallmatrix}\right]
\end{equation}
Setting $N:=u^2-x^2\pElem$, I compute:
\begin{equation}
\begin{split}
\bigl((x+y)\pElem&+u^{-1}(1+x(x+y)\pElem)\eElem\bigr)(u+x\eElem)\\
 &=\left((x+y)\pElem+u^{-1}(1+x(x+y)\pElem)\eElem\right)(u-x\eElem)N^{-1}\\
&= ux\pElem N^{-1}+uy\pElem N^{-1}+\eElem N^{-1}- u^{-1}x(1+x(x+y)\pElem)\pElem N^{-1}\\
&= y(u-u^{-1}x^2)\pElem N^{-1}+x(u- u^{-1}(1+x^2\pElem))\pElem N^{-1}+\eElem N^{-1}, \\
\end{split}
\end{equation}

so by the hypothesis $N\equiv1\mod{ \pIdeal{}^n}$ and $u\Id{2}+x\matP_{\Field{E}}\in\grp{T}_{\trZeroMat}^{(n)}$.

\begin{equation}
\begin{split}
 y(u-u^{-1}x^2)\pElem N^{-1}&+x(u- u^{-1}(1+x^2\pElem))\pElem N^{-1}+\eElem N^{-1} \\
 &\equiv y(u-u^{-1}x^2)\pElem+\eElem \quad \mod{ \pIdeal{}^n},
\end{split}
\end{equation}
since $\Torus_{\alpha,n}$ stabilizes $\smVertex{1}{\eElem}{n}$.  As $(u-u^{-1}x^2)$ is a unit, $y\pElem \equiv 0 \mod{ \pIdeal{}^n}$.  Therefor \linebreak[0] $u\Id{2}+x\matP\equiv\Mat{A}
\mod{\J{2n-1}^{\mathtt{r}}}.$\\
The corresponding calculation for $\val_{\Field{E}}(\diff{\pElem_{\Field{E}}})\leq \ramInd{\Field{E}}$ follows analogously.  
\indent In this situation, select $\eElem:=\pm\sqrt{1-4\pElem^{1-2d}}$ so that $4\pElem^{1-2d}=(1-\eElem)(1+\eElem)$. Here let $\trZeroMat=\left(\begin{smallmatrix}1&-2\pElem^{-d}\\ 2\pElem^{1-d}&-1\end{smallmatrix}\right)$ and $\Mat{A}:=\left(\begin{smallmatrix}x+y&-\frac{2y}{\pElem^d}\\ (z+\frac{2y}{\pElem^d})\pElem&(1-\frac{2y}{\pElem^d}(z+\frac{2y}{\pElem^d})\pElem)(x+y)^{-1}\end{smallmatrix}\right)$.  Assume $\Mat{A}\smVertex{1}{\frac{(\eElem+1)\pElem^d}{2}}{n\ramInd{\Field{E}/\Field{F}}}=\smVertex{1}{\frac{(\eElem+1)\pElem^d}{2}}{n\ramInd{\Field{E}/\Field{F}}}$.  Set $N=x^2+y^2\eElem^2$ and $y^\prime=\frac{2y}{\pElem^d}$.
\begin{equation}
\begin{split}
(x-&y\eElem)^{-1}\bigl((\eElem+1)\frac{\pElem^d}{2}\frac{1-y^\prime(z+y^\prime)\pElem}{x+y}+(z+y^\prime)\pElem\bigr)\\
&=N^{-1}\bigl((x+y+y(\eElem-1))(\eElem+1)\frac{\pElem^d}{2}\frac{1-y^\prime(z+y^\prime)\pElem}{x+y}\\
&\quad\quad\quad\quad\quad+(x-y+y(\eElem+1))(z\pElem+y^\prime\pElem)\bigr)\\
&=N^{-1}\bigl((\eElem+1)\frac{\pElem^d}{2}-y^\prime\pElem\frac{1+y^2-x^2-(y^\prime)^2\pElem}{x+y}\\
&\quad\quad\quad\quad\quad+z\pElem\frac{(y^\prime)^2\pElem +x^2-y^2}{x+y}\bigr)\\
&=N^{-1}\bigl((\eElem+1)\frac{\pElem^d}{2}-y^\prime\pElem\frac{1-N}{x+y}+z\frac{N\pElem}{x+y}\bigr)\\
&\equiv\bigl((\eElem+1)\frac{\pElem^d}{2}+z\frac{\pElem}{x+y}\bigr).
\end{split}
\end{equation}
by similar reasoning as above.  Again  $z\pElem \equiv 0 \mod{ \pIdeal{}^n}$, so $\Mat{A}\in\Torus_{\mElem, n}\grp{J}_{2n-1}^{\mathtt{r}}$.\\
\indent The unramified case proceeds where $\trZeroMat:=\left(\begin{smallmatrix}0& 1\\u&0\end{smallmatrix}\right)$ where $\Field{E}$ is the splitting field of $x^2-u$.  I omit the calculation as it is straightforward.\\
\indent This completes the proof of \autoref{lem:Tan} and consequently the proof of the theorem.\end{proof}

\section*{Appendix}\label{sec:App}
\begin{lemma}\label{lem:Trees}
For a finite extension $\Field{E}/\Field{F}$ there is a topological injection of trees $\Tree_{\Field{F}}\hookrightarrow\Tree_{\Field{E}}$, induced by tensoring lattices in $\Field{F}\times\Field{F}$ with $\oRing{\Field{E}}$ with these properties:
\begin{enumerate}[i.]
\item  As a map of metric spaces under the graph metric, this injection has a constant dilation factor of $\ramInd{\Field{E}/\Field{F}}$,
\item the natural action of $\GL{2}{\Field{F}}$ on $\Tree_{\Field{F}}$ extends to the natural action of $\GL{2}{\Field{F}}$ on $\Tree_{\Field{E}}$ as a subgroup of $\GL{2}{\Field{E}}$, and
\item the Galois action on $\Tree_{\Field{E}}$ is isometric, and stabilizes at least the image $\Tree_{\Field{F}}$.
\end{enumerate}
\end{lemma}
\begin{proof}\footnote{This proof is written for arbitrary $p$, for this paper $p=2$.  Moreover, if one considers the tree to be injected into Berkovich space these claims are trivial.  This proof is written assuming only knowledge of the definition of the Bruhat-Tits Tree found in \cite{S1977}. }
To establish the dilation factor, it will suffice to show that an arbitrary edge of $\Tree_{\Field{F}}$ maps to a path of length $\ramInd{\Field{E}/\Field{F}}$ in $\Tree_{\Field{E}}$.\\
Let $\mathcal{L}$ and $\mathcal{L}^\prime$ denote $\oRing{\Field{F}}$-latices in $\Field{F}\times\Field{F}$.\\
Each edge in $\Tree_{\Field{F}}$ has a representetive injection $\mathcal{L}\hookrightarrow\mathcal{L}^\prime$ where $\mathcal{L}^\prime/\mathcal{L}\simeq\oRing{\Field{F}}/\pIdeal{\Field{F}}=\F{q}$ for $q$ a power of $p$.  Tensoring these lattices by $\oRing{E}$ results in an injection $\mathcal{L}_{\Field{E}}\hookrightarrow\mathcal{L}_{\Field{E}}^\prime$.  The resulting injection is representative of a path in $\Tree_{\Field{E}}$.\\
I compute the length of this path.  Since $\mathcal{L}^\prime/\mathcal{L}\simeq\oRing{\Field{F}}/\pIdeal{\Field{F}}$, the quotient $\mathcal{L}_{\Field{E}}^\prime/\mathcal{L}_{\Field{E}}$ is isomorphic to $\oRing{\Field{E}}/(\pIdeal{\Field{E}}^{\ramInd{\Field{E}/\Field{F}}})$.  Hence, the distance between the vertexes of $\mathcal{L}_{\Field{E}}^\prime$ and $\mathcal{L}_{\Field{E}}$ is $\ramInd{\Field{E}/\Field{F}}$. \\ 
To prove the extendability of the action on $\Tree_{\Field{F}}$ to $\Tree_{\Field{E}}$, I point out that the action is determined by the action on the $\oRing{\Field{F}}$-generators of $\mathcal{L}$.  Meanwhile these generators are also $\oRing{\Field{F}}$-generators of $\mathcal{L}_{\Field{E}}$.  Hence the action of $\GL{2}{\Field{F}}$ extends to its action on $\Tree_{\Field{E}}$ as a subgroup of $\GL{2}{\Field{E}}$.\\
The isometry of conjugation follows from its conservation of valuation.  This implies that conjugation also conserves the indices of injections of lattices, which determine the distances between lattices.  That it stabilizes the image follows from the construction of the injection via tensoring with  $\oRing{\Field{E}}$.\\
That the image of $\Tree_{\Field{F}}$ is fixed by the Galois action is immediate.  \end{proof}
\begin{lemma}\label{lem:Eisnstn}
If $\Field{E}/\Field{F}$ is a ramified, quadratic extension, then $\oRing{\Field{E}}=\oRing{\Field{F}}[\pElem_{\Field{E}}]$ where $\pElem_{\Field{E}}$ is a root of an Eisenstein polynomial of the form,
$x^2+ax+b\pElem_{\Field{F}}$, where $b\in\Units_{\Field{F}}$, and either $a=$ or $a\in\pIdeal{\Field{F}}^{v}\setminus\pIdeal{\Field{F}}^{v+1}$, with $2v=\val_\Field{E}(\diff{\pElem_{\Field{E}}})\leq\ramInd{\Field{F}}$.
\end{lemma}
\begin{proof}[Proof of \autoref{lem:Eisnstn}]
The ring of integers of $\Field{E}$ is generated by its local uniformizing parameter (l.u.p.) since $[\Field{E}:\Field{F}]=[\pIdeal{\Field{E}}:\pIdeal{\Field{F}}]=2$.  Such a parameter, $\pElem_{\Field{E}}$ is a root of an Eisenstein polynomial as above where $a=\relTrace{}{\pElem_{\Field{E}}}$.  If there is a trace-zero l.u.p., then $a=0$ will suffice.\\
Assume that there is a l.u.p., $\pElem_{\Field{E}}^\prime$ with trace of valuation greater than $\ramInd{\Field{F}}$.  Then $\relTrace{}{\pElem_{\Field{E}}^\prime}\!/{2}\in\pIdeal{\Field{F}}$.  Whence, $\pElem_{\Field{E}}:= \pElem_{\Field{E}}^\prime-\relTrace{}{\pElem_{\Field{E}}^\prime}\!/{2}$ is a l.u.p. of $\Field{E}$ and has trace zero. \\ 
Assume that no l.u.p. in $\Field{E}$ is trace-zero.  Any trace zero then element has even valuation.  Any l.u.p., $\pElem_{\Field{E}}$, in $\Field{E}$ then has a trace of valuation no greater than $\ramInd{\Field{F}}$.  Since $\relTrace{}{\pElem_{\Field{E}}}\equiv \diff{\pElem_{\Field{E}}}\mod {2\pIdeal{\Field{E}}}$,  $\val\bigl(\diff{\pElem_{\Field{E}}}\bigr)=\val\bigl(\relTrace{}{\pElem_{\Field{E}}}\bigr)$.
\end{proof}
\clearpage
\bibliographystyle{plain} 
{\bibliography{Stab}}
\end{document}